\documentclass[10pt,a4paper]{article}
\usepackage[utf8x]{inputenc}
\usepackage{amsmath, amsthm}
\usepackage{amsfonts}
\usepackage{amssymb}
\usepackage{graphicx}
\usepackage{pgf,tikz}
\usetikzlibrary{arrows}
\usepackage{enumerate}

\newtheorem{theorem}{Theorem}[section]
\newtheorem{defin}[theorem]{Definition}
\newtheorem{lemma}[theorem]{Lemma}

\newtheorem{prop}[theorem]{Proposition}

\DeclareMathOperator{\HF}{HF}
\DeclareMathOperator{\HS}{HS}
\DeclareMathOperator{\rad}{r}
\DeclareMathOperator{\rank}{rank}
\DeclareMathOperator{\Syz}{Syz}
\DeclareMathOperator{\vdim}{vdim}

\begin{document}
\title{The strong Lefschetz property of monomial complete intersections in two variables}
\author{Lisa Nicklasson}
\date{}

\maketitle

\begin{abstract}
 In this paper we classify the monomial complete intersection algebras, in two variables, and of positive characteristic, which has the strong Lefschetz property. Together with known results, this gives a complete classification of the monomial complete intersections with the strong Lefschetz property.
\end{abstract}

\section{Background}

A graded algebra $A=\bigoplus_{i \ge 0} A_i$ is said to have the \emph{strong Lefschetz property} (SLP) if there is a linear form such that multiplication by any power of this linear form has maximal rank in every degree. Let $A$ be a monomial complete intersection, that is $A=K[x_1, \ldots, x_n]/(x_1^{d_1}, \ldots, x_n^{d_n})$, where $K$ is a field and $d_1, \ldots, d_n$ some positive integers. In characteristic zero, $A$ always has the SLP, which was first proved by Stanley in \cite{Stanley}. When the characteristic is positive, the algebra does not always have the SLP. A first result is that $A$ has the SLP when $p>\sum (d_i-1)$, where $p$ is the characteristic. This was proved in the case $n=2$ by Lindsey in \cite{Lindsey}, and later in the general case by Cook II in \cite{CookII}. 

A classification of all monomial complete intersections in three or more variables with the SLP is provided in \cite{Lundqvist-Nicklasson}. Notice that the problem is trivial when $n=1$, so the remaining case is $n=2$, which will be treated in this paper. The sufficient conditions in \cite{Lundqvist-Nicklasson} hold also in two variables, but it turns out that there is an additional class of algebras $K[x,y]/(x^a,y^b)$ with the SLP. This is indicated by Cook II in \cite{CookII}, where the two special cases, when $a=b$, and when the characteristic is two, is studied. Cook II solves these cases, under the assumption that the residue field $K$ is infinite. 

The main result of this paper is Theorem \ref{thm:slp_twovar_p>2}, which is a classification of the algebras $K[x,y]/(x^a,y^b)$ with the SLP, where $K$ is a field of characteristic $p \ge 3$. The classification is given in terms of the base $p$ digits of the integers $a$ and $b$. Together with the mentioned earlier results, this gives a complete classification of the monomial complete intersections with the SLP, see Theorem \ref{thm:slp_classification}. 

The technique used both in \cite{CookII} and in this paper, is the theory of the syzygy gap function, introduced by Monsky in \cite{Monsky}. The syzygy gap function deals with the degrees of the relations on $x^a, y^b$ and $(x+y)^c$. This can then be connected to the SLP using results of Brenner and Kaid in \cite{Brenner-Kaid-syzbundle} and \cite{Brenner-lookingout}. In \cite{Brenner-Kaid-syzbundle}, \cite{Brenner-lookingout}, and \cite{Monsky} the residue field is required to be algebraically closed. We will see in Section \ref{sec:syz-gap} that this assumption can be dropped. We will also give a new proof of the connection to the SLP, when working with monomial complete intersections.  

\section{The strong Lefschetz property}
 Let $A=\bigoplus_{i \ge 0} A_i$ be a graded algebra. A linear map $A_i \to A_{j}$ is said to have \emph{maximal rank} if it is injective or surjective. Each homogeneous element $f \in A_d$ induces a family of linear maps $A_i \to A_{i+d}$ by $a \mapsto f \cdot a$. Let such maps be denoted by $\cdot f: A_i \to A_{i+d}$. For short, we say that multiplication by $f$ has \emph{maximal rank in every degree}, if all the maps induced by $f$ have maximal rank.

\begin{defin} A graded algebra $A$ is said to have the \emph{strong Lefschetz property} (SLP) if there exists an $\ell \in A_1$ such that the maps $\cdot \ell^m: A_i \to A_{i+m}$ have maximal rank for all $i\ge 0$ and all $m\geq 1$. In this case, $\ell$ is said to be a \emph{strong Lefschetz element}.

We say that $A$ has the \emph{weak Lefschetz property} (WLP) if there exists an $\ell \in A_1$ such that the maps $\cdot \ell: A_i \to A_{i+1},$ have maximal rank for all $i\ge 0$. In this case, $\ell$ is said to be a \emph{weak Lefschetz element}.
\end{defin}

Let now $K$ be a field, and $A=K[x_1, \ldots, x_n]/I$, where $I$ is a monomial ideal. In \cite[Proposition 4.3]{Lundqvist-Nicklasson} it is proved that $A$ has the WLP if and only if $x_1+ \dots + x_n$ is a weak Lefschetz element. The corresponding is also true for the strong Lefschetz property. 

\begin{theorem}\label{thm:lef_element}
Let $R=K[x_1, \ldots, x_n]$, where $K$ is a field, and let $I \subset R$ be a monomial ideal. Then $R/I$ has the SLP (WLP) if and only if $x_1 + \ldots + x_n$ is a strong (weak) Lefschetz element.  
\end{theorem}
\begin{proof}
Suppose that $\sum_{i \in \Lambda}c_ix_i$, for some $\Lambda \subseteq \{1, \ldots, n\}$ and $0 \ne c_i \in K$, is a strong Lefschetz element of $A=R/I$. The monomial ideal $I$ is left unchanged under a change of variables $c_ix_i \mapsto x_i$. This shows that $\ell=\sum_{i \in \Lambda}x_i$ also is a strong Lefschetz element. If $\Lambda=\{1, \ldots, n\}$ we are done. Assume that $\Lambda\subset \{1, \ldots, n\}$, and $j \notin \Lambda$. The next step is to prove that $x_j+\ell$, is also a strong Lefschetz element. For this purpose we introduce a new element $a$ in an extension field of the type $K'=K(a) \supset K$. We will prove that $ax_j+\ell$ is a strong Lefschetz element in $A'=A\otimes_K K'$. Let, for each $i$, $B_i$ be the vector space basis for $A_i$ that consists of monic monomials. This is also a basis for $A'_i$, as a vector space over $K'$. Let $M$ be the matrix of the multiplication map 
\[
 \cdot (ax_j+\ell)^m: A'_i \to A'_{i+m},
\]
w. r. t. the bases $B_i$ and $B_{i+m}$. The entries of $M$ are polynomials in $a$. Let $M_0$ be the matrix we obtain by substituting $a=0$ in $M$. If $M$ does not have maximal rank, neither does $M_0$. But $M_0$ is the matrix of the map $ \cdot \ell^m: A_i \to A_{i+m}$, which has maximal rank. This shows that $M$ has maximal rank, and $ax_j+\ell$ is a strong Lefshetz element of $A'$. But then, since $a$ is a non-zero element of the field $K'$, so is $x_j+\ell$. The coefficients of $x_j+\ell$ are in $K$, so it is also a strong Lefschetz element of $A$. It follows that $x_1 + \dots + x_n$ is a strong Lefschetz element of $A$.    
\end{proof}

The \emph{Hilbert function} of a graded algebra $A=\bigoplus_{i \ge 0}A_i$ with residue field $K$ is a function $\HF_A:\mathbb{Z}_{\ge 0} \to \mathbb{Z}_{\ge 0}$ definied by $\HF_A(i)=\vdim_K A_i$, i. e. the vector space dimension of $A_i$ over $K$. The \emph{Hilbert series} of $A$, denoted $\HS_A$, is the generating function of the sequence $\HF(i)$, that is $\HS_A(t)=\sum_{i \ge 0} \HF(i)t^i$.

Let now $A$ be a monomial complete intersection, $A=K[x_1, \ldots, x_n]/(x_1^{d_1}, \ldots, x_n^{d_n})$, for some positive integers $d_1, \dots, d_n$. Let $t=\sum_{i=1}^n(d_i-1)$. This is the highest possible degree of a monomial in $A$, and hence $\HF_A(i)=0$ when $i>t$. It can also be seen that the Hilbert function is symmetric about $t/2$, and that $\HF_A(i) \le \HF_A(i+d)$ when $i \le (t-d)/2$. For a multiplication map to have maximal rank in every degree in $A$, it shall then be injective up to some degree $i$, and surjective for larger $i$. It can be proved that the injectiveness in this case implies the surjectiveness. 

\begin{prop}\label{prop:maxrang-inj} 
Let $A=K[x_1, \ldots, x_n]/(x_1^{d_1}, \ldots, x_n^{d_n})$ and $t=\sum_{i=1}^n(d_i-1)$, and let $f\in A$ be a form of degree $d$. The maps $\cdot f:A_i \to A_{i+d}$ all have maximal rank if and only if the maps with $i \le (t-d)/2$ are injective. 
%

\end{prop}
\begin{proof}
 See e. g. \cite[Proposition 2.6]{Lundqvist-Nicklasson}.
\end{proof}
In other words, multiplication by a form $f$ has maximal rank in every degree if all homogeneous zero divisors of $f$ are of degree greater than $(t-d)/2$. Another interesting fact is that if we consider forms of the type $\ell^d$, and $t-d$ is even, then multiplicaton by $\ell^{d+1}$ has maximal rank in every degree if multiplication by $\ell^d$ does. This result will be important for the classification of algebras with the SLP when $n=2$. 

\begin{prop}\label{prop:every_other_lefschetzel}
 Let $A=K[x_1, \ldots, x_n]/(x_1^{d_1}, \ldots, x_n^{d_n})$ and $t=\sum_{i=1}^n(d_i-1)$. Let $\ell\in A$ be a linear form, and $d$ a positive integer such that $t-d$ is even. If the maps $\cdot \ell^d: A_{i} \to A_{i+d}$ have maximal rank for all $i\ge 0$, so does the maps $\cdot \ell^{d+1}: A_{i} \to A_{i+d+1}.$
\end{prop}

\begin{proof}
Assume that $\cdot \ell^d: A_{i} \to A_{i+d}$ have maximal rank for all $i\ge 0$. By Proposition \ref{prop:maxrang-inj} all zero divisors of $\ell^d$ are of degree at least $(t-d)/2$. Suppose that there is a homogeneous element $f$ such that $\ell^{d+1}f=0$. By Proposition \ref{prop:maxrang-inj}, we are done if we can prove that $\deg(f)>(t-(d+1))/2=(t-d)/2-1/2$. Since $t-d$ is even, the right hand side is not an integer, and it is enough to prove $\deg(f)>(t-d)/2-1$. 
 Consider first the case when $\ell^df=0$. That is, $f$ is a zero divisor of $\ell^d$, and it follows that  $\deg(f)>(t-d)/2$. Consider instead the case when $\ell^d f \ne 0$. We know that $\ell^{d+1}f=0$, that is $\ell f$ is a homogeneous zero divisor of $\ell^d$. Then $\deg(\ell f)>(t-d)/2$, and $\deg(f)>(t-d)/2-1$, which finishes the proof.
\end{proof}

\begin{prop}\label{prop:slp-maxrang}
The algebra $A=K[x,y]/(x^a,y^b)$ has the SLP if and only if the maps 
\[
 \cdot (x+y)^{a+b-2c}:A_i \to A_{i+a+b-2c}
\]
have maximal rank for all $i \ge 0$ and $1 \le c < \min(a,b)$.
\end{prop}
\begin{proof}
The ''only if''-part follows from Theorem \ref{thm:lef_element}. 

The numbers $t$ and $d$ in Proposition \ref{prop:every_other_lefschetzel} are here $t=a+b-2$, and $d=a+b-2c$. We see that $t-d=2c-2$ is even, so if multiplication by $(x+y)^{a+b-2c}$ has maximal rank in every degree, so does multiplication by $(x+y)^{a+b-2c+1}$. If $c\le0$ then $A_{i+a+b-2c}=\{0\}$, and obviously any map $A_i \to A_{i+a+b-2c}$ is surjective. This is why we only need to consider $c \ge 1$.
 Without loss of generality, we can assume that $a=\min(a,b)$. To complete the proof we need to show that multiplication by $(x+y)^{a+b-2c}$ has maximal rank in every degree when $c \ge a$. Suppose there is a non-zero homogenous $f \in A$ such that $(x+y)^{a+b-2c}f=0$. By Proposition \ref{prop:maxrang-inj} multiplication by $(x+y)^{a+b-2c}$ has maximal rank in every degree if we can prove that
 \[
  \deg(f)>\frac{a+b-2-(a+b-2c)}{2}=c-1.
 \]
 Let $F$ be a homogeneous element in $K[x,y]$ whose image in $A$ is $f$. Then
 \[
  (x+y)^{a+b-2c}F = gx^a+hy^b, ~~ \textrm{for some} ~ g, h \in K[x,y].
 \]
We can not have $h=0$, because that would imply that $F$ is divisible by $x^a$, and $f=0$ in $A$. Hence $h \ne 0$ and   $\deg((x+y)^{a+b-2c}F) \ge b$, which is equivalent to $\deg(F) \ge 2c-a.$ If $c \ge a$ this implies $\deg(f)=\deg(F) \ge c$, and we are done.
\end{proof}

\section{Classifying the monomial complete intersections with the strong Lefschetz property}
A classification of the monomial complete intersections with the SLP, in three or more variables, is given in \cite[Theorem 3.8]{Lundqvist-Nicklasson}. Here we give a slightly reformulated version of the theorem, to make the notation similar to that used later in the case of two variables. We will prove that the formulation here is equivalent to that in \cite{Lundqvist-Nicklasson}. 

\begin{theorem}\label{thm:slp_n>2}
Let $A=K[x_1, \ldots, x_n]/(x_1^{d_1}, \ldots, x_n^{d_n})$ where $n \ge 3$, $d_i \ge 2$ for all $i$, and $K$ is a field of characteristic $p>0$. Let $t = \sum_{i = 1}^n (d_i-1)$ and
let $d_1=\max(d_1, \dots, d_n)$. Write $d_1=N_1p+r_1$ with $0\le r_1< p$.
Then $A$ has the SLP if and only if one of the following two conditions hold
\begin{enumerate}
\item $t <p$,
\item $d_1 \ge p$, $d_i < p$ for $i =2, \ldots, n$ and $\sum_{i=2}^n(d_i-1) \le \min(r_1,p-r_1)$.
\end{enumerate} 
\end{theorem}
\begin{proof}
 The difference, compared to \cite[Theorem 3.8]{Lundqvist-Nicklasson}, is that in \cite{Lundqvist-Nicklasson} the bound for $r_1$ is $0< r_1 \le p$, and the second condition is 
 \[
   d_1 > p,~ d_i \le p~ \mbox{for} ~ i =2, \ldots, n~ \mbox{and} ~ \sum_{i=2}^n(d_i-1) \le \min(r_1,p-r_1).
 \]
It is easy to see that both definitions of $r_1$ gives the same value $\min(r_1,p-r_1)$. When $d_1=p$ condition 2 of \cite[Theorem 3.8]{Lundqvist-Nicklasson} is not satisfied. Neither is condition 2 in Theorem \ref{thm:slp_n>2}, because $\min(r_1,p-r_1)=0$, and $\sum_{i=2}^n(d_i-1) \ge n-1 \ge 2$. When $d_i=p$, for some $i>1$, condition 2 in Theorem \ref{thm:slp_n>2} is not satisfied. Neither is 2 in \cite[Theorem 3.8]{Lundqvist-Nicklasson}, because then $\sum_{i=2}^n(d_i-1) \ge p$, and $\min(r_1,p-r_1)<p$ in general. This shows that both formulations agree. 
\end{proof}

The two conditions in Theorem \ref{thm:slp_n>2} above can be genaralized to the case $n=2$. Next we will prove that in two variables, and characteristic $p>2$, the algebra $A$ has the SLP in these two cases, but also in an additional one. 

\begin{theorem}\label{thm:slp_twovar_p>2}
 Let $A=K[x,y]/(x^a,y^b)$, where $a,b \ge 2$ and $K$ is a field of characteristic $p>2$. Write $a$ and $b$ in base $p$, that is $a=a_kp^k + \dots + a_1p+a_0$ and $b=b_\ell p^\ell + \dots + b_1p+b_0$, where $0 \le a_i,b_i <p$, and $a_k,b_\ell \ne 0$. We may assume that $ \ell \ge k$. The classification of the algebras with the SLP is divided into three cases.
 \begin{enumerate}
  \item When $a,b<p$, $A$ has the SLP if and only if $a+b \le p+1$. 
 
 \item When $a<p$ and $b \ge p$, $A$ has the SLP if and only if $a \le \min(b_0,p-b_0) +1$. 
 
 \item When $a,b \ge p$, $A$ has the SLP if and only if the following three conditions are satisfied. 
 \begin{enumerate}[(a)]
 \item $a_0= \frac{p\pm 1}{2}$, and $b_0=\frac{p\pm 1}{2}$, 
 \item $a_i=b_i=\frac{p-1}{2}$ for $i=1, 2, \ldots, k-1$,
 \item $a_k+b_k \le p-1$, and $b_k \ge a_k$ when $\ell >k$.
\end{enumerate}\end{enumerate}
\end{theorem}
Notice that there are no restrictions on $b_i$ for $i>k$, in the case $\ell >k$. The theorem will be proved later in this section. 

In \cite[Theorem 4.9]{CookII} Cook II proves the special case $a=b$ of Theorem \ref{thm:slp_twovar_p>2}. Cook II also proves the characteristic two case.

\begin{theorem}[{\cite[Corollary 4.8]{CookII}}]\label{thm:slp_twovar_p=2}
 Let $A=K[x,y]/(x^a,y^b)$, where $ 2 \le a \le b $ and $K$ is a field of characteristic two. $A$ has the SLP if and only if one of the two following conditions hold.
 \begin{enumerate}
  \item $a=2$ and $b$ is odd,
  \item $a=3$ and $b \equiv 2 \mod 4$.
 \end{enumerate}
\end{theorem}

Theorem \ref{thm:slp_n>2}, Theorem \ref{thm:slp_twovar_p>2}, and Theorem \ref{thm:slp_twovar_p=2} can now be combined into a complete classification of the monomial complete intersections with the SLP. 

\begin{theorem}\label{thm:slp_classification}
 Let $A=K[x_1, \ldots, x_n]/(x_1^{d_1}, \ldots, x_n^{d_n})$, where all $d_i \ge 2$ and $K$ is a field of characteristic $p>0$. Write each $d_i$ in base $p$ as $d_i=c_{ik_i}p^{k_i} + \dots + c_{i1}p+c_{i0}$, with $c_{ik_i} \ne 0$. The algebra $A$ has the SLP if and only if one of the following conditions hold.
 \begin{enumerate}
  \item $n=1$,
  \item $n=2$, $p=2$, and one of the following holds, for $d_1 \le d_2$
  \begin{itemize}
   \item $d_1=2$ and $c_{20}=1$,
   \item $d_1=3, c_{21}=1$, and $c_{20}=0$,
  \end{itemize}
  \item $n=2$, $p>2$ and all the following conditions are satisfied, for $k_1 \le k_2$
  \begin{itemize}
   \item $c_{10}=\frac{p \pm 1}{2}, c_{20}=\frac{p \pm 1}{2}$,
   \item $c_{1j}=c_{2j}=\frac{p-1}{2}$, for $j=1, \ldots, k_1-1$,
   \item $c_{1k_1}+c_{2k_1} < p$, and $c_{2k_1} \ge c_{1k_1}$ if $k_1<k_2$,
  \end{itemize}
  \item $n \ge 2$, and $\sum_{i=1}^n (d_i-1) <p$,
  \item $n \ge 2$, and there is a $j$ such that ${d_j \ge p}$, ${d_i<p}$ for all ${i \ne j}$, and ${\sum_{i \ne j}(d_i-1) \le \min(c_{j0},p-c_{j0})}$.
  \end{enumerate}
\end{theorem}
\begin{proof}
 The case $n=1$ is trivial. Condition 3 is condition 3 of Theorem \ref{thm:slp_twovar_p>2}, and Condition 4 is Theorem \ref{thm:slp_twovar_p=2} with $b=d_2$ written in base 2. The conditions 4 and 5 are the conditions 1 and 2 from Theorem \ref{thm:slp_n>2} and Theorem \ref{thm:slp_twovar_p>2} combined. Notice that 4 and 5 are not satisfied when $p=2$. 
\end{proof}

 Both proofs of \cite[Corollary 4.8]{CookII} and \cite[Theorem 4.9]{CookII} uses Theorem \ref{thm:slp_manhattan} below. This will also be the key to the proof of Theorem \ref{thm:slp_twovar_p>2}.

\begin{theorem}\label{thm:slp_manhattan}
Let $K$ be a field of characteristic ${p>0}$. The algebra $K[x,y]/(x^a,y^b)$ has the SLP if and only if
\[|a-up^i|+|b-vp^i|+|a+b-2c-wp^i| \ge p^i \]
for all integers $i \ge 0$, $1 \le c <\min(d_1,d_2)$, and $u,v,w$ such that $u+v+w$ is odd. 
\end{theorem}

Theorem \ref{thm:slp_manhattan} is proved in Section \ref{sec:syz-gap}.

We will now prove that Theorem \ref{thm:slp_manhattan} can be reformulated as the following proposition. 

\begin{prop}\label{prop:slp_step}
 Let $A=K[x,y]/(x^a,y^b)$, where $K$ is a field of characteristic $p>0$. For each integer $i \ge 1$ we can write $a=m_ip^i+r_i$, and $b=n_ip^i+s_i$, where $0 \le r_i,s_i <p^i$. The algebra $A$ has the SLP if and only if the following conditions hold for all $i$. 
 \begin{enumerate}
  \item If $m_i>0$, then $r_i \ge s_i-1$,
  \item If $n_i>0$, then $s_i \ge r_i-1$,
  \item If $m_i >0$ and $n_i>0$, then $r_i+s_i \ge p^i-1$,
  \item $r_i+s_i \le p^i+1$.
 \end{enumerate}
\end{prop}
\begin{proof}
We shall prove that the conditions above is equivalent to that in Theorem \ref{thm:slp_manhattan}. 
Let us investigate for which $a$ and $b$ it can happen that 
\[
 |a-up^i| + |b-vp^i| + |a+b-2c-wp^i| < p^i.
 \]
Write $a=m_ip^i + r_i$ and $b=n_ip^i+s_i$, as in the proposition. Notice that 
\[
 |a-up^i| = \left\{ \begin{array}{ll}
                     r_i & \textrm{when } u=m_i \\
                     p^i-r_i & \textrm{when } u=m_i+1.
                    \end{array} \right.
\]
For all other values of $u$ we get $|a-up^i|\ge p^i$, and then of course $|a-up^i| + |b-vp^i| + |a+b-2c-wp^i|\ge p^i$. Therefore we only need to consider $u=m_i$ and $u=m_i+1$. The corresponding is also true for $|b-vp^i|$. This gives us four cases to examine. 
\begin{enumerate}[{\bf I.}]
\item $u=m_i$ and $v=n_i$\\
Here 
\[
 |a-up^i| + |b-vp^i| = r_i+s_i.
\]
To obtain $|a-up^i| + |b-vp^i| + |a+b-2c-wp^i|\le p^i-1$ it is necessary that $r_i+s_i \le p^i-1$. 

Suppose first that $r_i+s_i=p^i-1$. Since $u+v+w=m_i+n_i+w$ is supposed to be odd, we must have $w=m_i+n_i -2d+1$, for some integer $d$. Then
\begin{align*}
 a+b-2c-wp^i &= n_ip^i + r_i + m_ip^i+s_i -2c -(m_i+n_i -2d+1)p^i\\
 &= r_i+s_i-2c  + (2d-1)p^i = 2dp^i-2c-1,
\end{align*}
which is an odd number, and thus $|a+b-2c-wp^i|\ge 1$. We get 
\[
 |a-up^i| + |b-vp^i| + |a+b-2c-wp^i| \ge p^i-1+1=p^i,
\]
and we can conclude that $|a-up^i| + |b-vp^i| + |a+b-2c-wp^i|\ge p^i$ for all $w$ and $c$, when $r_i+s_i=p^i-1$.

Now suppose that $r_i+s_i \le p^i-2$. We want to find out what the smallest possible value of $|a+b-2c-wp^i|$ is. For this purpose we choose the largest $w$ such that $u+v+w$ is odd, and $a+b-wp^i>0$. After that we choose the value for $c$ that makes $|a+b-2c-wp^i|$ as small as possible. Since $r_i+s_i \le p^i-2$, the largest $w$ with the required properties is $w=m_i+n_i-1$. Then
\[
 a+b-wp^i = p^i+r_i+s_i.
\]
If $m_i=0$, then $\min(a,b) = \min(r_i,b) \le r_i$ and $c \le r_i-1$. Then
\[
 a+b-2c-wp^i \ge p^i+r_i+s_i -2(r_i-1) = p^i-r_i+s_i+2, ~~\textrm{and}
\]
\[
 |a-up^i| + |b-vp^i| + |a+b-2c-wp^i| \ge r_i+s_i + p^i-r_i+s_i+2>p^i.
\]
In a similar way we see that $|a-up^i| + |b-vp^i| + |a+b-2c-wp^i|>p^i$ if $n_i=0$. Suppose now that $m_i>0$ och $n_i>0$. Then we choose $c=[(p^i+r_i+s_i)/2]$, where $[. . .]$ denotes the integer part. This gives 
\[
 a+b-2c-wp^i = 0 ~ \textrm{or} ~ 1, ~~\textrm{and}
\]
\[
 |a-up^i| + |b-vp^i| + |a+b-2c-wp^i| \le r_i + s_i +1 \le p^i -1. 
\]
The conclusion, in this case, is that $|a-up^i| + |b-vp^i| + |a+b-2c-wp^i|<p^i$, exactly when $m_i,n_i>0$ and $r_i+s_i \le p^i-2$. This corresponds to condition 3 in the proposition. 

\item $u=m_i$ and $v=n_i+1$\\
Here 
\[
 |a-up^i| + |b-vp^i| = r_i+p^i-s_i.
\]
To obtain $|a-up^i| + |b-vp^i| + |a+b-2c-wp^i|\le p^i-1$ it is necessary that $r_i+p^i-s_i \le p^i-1$, that is $r_i \le s_i-1$. Let us first consider the case when $r_i=s_i-1$. Since $u+v+w$ is supposed to be odd we must have $w=n_i+m_i -2d$, for some integer $d$. This gives
\[
 a+b-wp^i = r_i+s_i+2dp^i = 2r_i+1 + 2dp^i,
\]
which is odd. Then $|a+b-2c-wp^i|\ge 1$, and 
\[
 |a-up^i| + |b-vp^i| + |a+b-2c-wp^i| \ge r_i+p^i-s_i+1 =p^i.
\]
Suppose instead that $r_i \le s_i-2$. We use that same idea as in case 1, and choose first $w$, and then $c$, such that $|a+b-2c-wp^i|$ has the smallest possible value. The best option for $w$ is $w=n_i+m_i$. This gives
\[
a+b-wp^i=r_i+s_i.
\]
If $m_i=0$, then $\min(a,b)=\min(r_i,b)=r_i$, thus $c=r_i-1$ is the largest allowed value of $c$. Then 
\[
 a+b-2c-wp^i = r_i+s_i-2(r_i-1) = s_i-r_i+2, ~~\textrm{and}
\]
\[
 |a-up^i| + |b-vp^i| + |a+b-2c-wp^i| = r_i+p^i-s_i + s_i-r_i+2=p^i+2.
\]
If $m_i>0$ on the other hand, we are allowed tho choose $c=s_i-1$. Then we get
\[
 a+b-2c-wp^i = r_i-s_i+2
\]
instead. Note that this is a non-positive number. This gives
\[
 |a-up^i| + |b-vp^i| + |a+b-2c-wp^i| = r_i+p^i-s_i + s_i-r_i-2=p^i-2.
\]
The conclusion, in this case, is that $|a-up^i| + |b-vp^i| + |a+b-2c-wp^i|<p^i$, exactly when $m_i>0$ and $r_i \le s_i-2$. This corresponds to condition 1 in the proposition. 

\item $u=m_i+1$ and $v=n_i$

In the same way as above, we see that this corresponds to condition 2. 

\item $u=m_i+1$ och $v=n_i+1$

Here
\[
|a-up^i| + |b-vp^i| = 2p^i-r_i-s_i,
\]
so for this to be smaller than $p^i$ we must have $2p^i-r_i-s_i \le p^i-1$, which is $r_i+s_i \ge p^i+1$. Consider first the case when $r_i+s_i=p^i+1$. Then we must choose $w=m_i+n_i-2d+1$, for some integer $d$. Then 
\[
 a+b-wp^i = r_i+s_i +(2d-1)p^i = 2dp^i+1,
\]
and $|a+b-2c-wp^i|\ge 1$. Then we get 
\[
 |a-up^i| + |b-vp^i| + |a+b-2c-wp^i| \ge 2p^i-r_i-s_i+1 = p^i.
\]
Suppose now that $r_i + s_i \ge p^i +2$. We choose $w=m_i+n_i+1$ and $c=[(r_i+s_i-p^i)/2]$, because this gives
\[
 a+b-2c-wp^i = r_i+s_i-p^i -2c = 0 ~ \textrm{or} ~ 1,~~ \textrm{and}
\]
\[
 |a-up^i| + |b-vp^i| + |a+b-2c-wp^i| \le 2p^i-r_i-s_i +1 \le p^i -1.
\]
This shows that $|a-up^i| + |b-vp^i| + |a+b-2c-wp^i|<p^i$ when $r_i + s_i \ge p^i +2$, which is condition 4.
\end{enumerate}
\end{proof}

Proposition \ref{prop:slp_step} will be used later in this section to prove Proposition \ref{prop:monomial_kernel}, which says something about the structure of an algebra that does not have the SLP. Now we shall use Proposition \ref{prop:slp_step}, with $p>2$, to prove Theorem \ref{thm:slp_twovar_p>2}. 

\begin{proof}[Proof of Theorem \ref{thm:slp_twovar_p>2}] Let $A=K[x,y]/(x^a,y^b)$, and suppose throughout this proof that the characteristic of $K$ is greater than 2. Write $a$ and $b$ in base $p$ as $a=a_kp^k + \dots + a_1p+a_0$ and $b=b_\ell p^\ell + \dots + b_1p+b_0$, where $0 \le a_i,b_i <p$. We assume that $ \ell \le k$. With the notation $a=m_ip^i+r_i$ from Proposition \ref{prop:slp_step} we have $r_i=a_{i-1}{p^{i-1}} + \dots + a_1p+a_0$, and $m_i=a_kp^{k-i} + a_{k-1}p^{k-i-1}+ \dots + a_i$, and similar for $b$.  
 
If $a,b < p$ then $n_i=m_i=0$ in Proposition \ref{prop:slp_step}, for all $i$, and the conditions 1, 2 and 3 are trivially satisfied. Since $a+b<2p$ condition 4 is satisfied for $i>1$. The only restriction we get comes from condition 4 when $i=1$, and states that $A$ has the SLP if and only if $a+b \le p+1$. 

If $a < p$ and $b \ge p$ we get $b_0 \ge a_0-1$ and $a_0+b_0 \le p+1$ from the conditions 2 and 4 with $i=1$. These two inequalities can be written as $a_0 \le \min(b_0, p-b_0)+1$. In condition 1 and 3 there is nothing to check, and for $i>1$ all conditions are satisfied. We get that $A$ has the SLP if and only if $a_0 \le \min(b_0, p-b_0)+1$.

Assume now that $a, b \ge p$. The idea now is to translate the four conditions of Proposition \ref{prop:slp_step} into the base $p$ digits of $a$ and $b$.

Let us first look at $i=1$ in Proposition \ref{prop:slp_step}. We know that $m_1,n_1>0$, so 1 and 2 gives $a_0-1 \le b_0 \le a_0+1$. The conditions 3 and 4 gives $p-1 \le a_0+b_0 \le p+1$. Both these intequalites are satisfied exactly when $a_0=\frac{p \pm 1}{2}$ and $b_0=\frac{p \pm 1}{2}$. This is condition (a) in Theorem \ref{thm:slp_twovar_p>2}. Suppose that this is the case, and move on to $i=2$. If $k \ge 2$ then $m_2$ and $n_2$ are positive. The conditions 1 and 2 gives
\[
 a_1p+a_0 -1 \le b_1p+b_0 \le a_1p+a_0+1,
\]
which implies $a_1=b_1$. For 3 and 4 to be satisfied
\[
 p^2-1 \le (a_1+b_1)p+(a_0+b_0) \le p^2+1
\]
is required. This is true if and only if $a_1+b_1=p-1$. Hence we get $a_1=b_1=\frac{p-1}{2}$. We suppose that this is true and continue with $i=3, \dots k$. In the same way as above we get $a_2= \dots =a_{k-1} = b_2 = \dots =b_{k-1} = \frac{p-1}{2}$. This is condition (b) in Theorem \ref{thm:slp_twovar_p>2}.

Suppose that the conditions for $i=1, 2 \ldots, k$ are satisfied, and move on to $i=k+1$. Now $m_{k+1}=0$, so in condition 1 and 3 there is nothing to check. If $\ell >k$ then $n_{k+1}>0$. In this case condition 2 says
\[
 b_kp^k + \dots + b_1p+b_0 \ge a_kp^k + \dots + a_1p+a_0-1,
\]
which holds if and only if $b_k \ge a_k$. Condition 4 says
\[
 (a_k+b_k)p^k + \dots + (a_1+b_1)p+(a_0+b_0) \le p^{k+1} +1,
\]
which holds if and only if  $a_k+b_k \le p-1$. This proves (c). 

We must also show that there are no futher restrictions on $b_j$ for $j>k$, when such $b_j$ exist. Suppose that the four conditions of Proposition \ref{prop:slp_step} are satisfied for $i=1, 2, \ldots, k+1$. We continue by looking at $i=k+2$. The conditions 1 and 3 are satisfied, since $m_i=0$. Notice also that $r_{k+2}=r_{k+1}=a$, and $s_{k+2} \ge s_{k+1}$. This means that if condition 2 is satisfied for $i=k+1$, so it is for $i=k+2$. Condition 4 requires
\[
 b_{k+1}p^{k+1} + (a_k+b_k)p^k + \dots +(a_1+b_1)p + (a_0+b_0) \le p^{k+2} +1.
\]
But this is no restriction on $b_{k+1}$, other than $b_{k+1}<p$. The same reasoning works for larger $i$. 
\end{proof}

The proof in \cite{Lundqvist-Nicklasson} of when an algebra in three or more variables does not have the SLP, is carried out by finding a monomial zero-divisor of $(x_1+ \dots + x_n)^m$, for some $m$. We will now see that this can also be done in two variables. This gives an alternative proof of the ''only if''-part of Theorem \ref{thm:slp_twovar_p>2}.

\begin{prop}\label{prop:monomial_kernel}
 Let $A=K[x_1, \ldots, x_n]/(x_1^{d_1}, \ldots, x_n^{d_n})=\bigoplus_{i \ge 0} A_i$ be an algebra of characteristic $p>0$ which does not possess the SLP. Let $\ell$ be a linear form in $A$. Then there are integers $d$ and $m$ such that $\HF_A(d) \le \HF_A(d+m)$, and the kernel of the multiplication map $\cdot \ell^m:A_d \to A_{d+m}$ contains a non-zero monomial.
\end{prop}
\begin{proof}
For the case $n \ge 3$, see \cite{Lundqvist-Nicklasson}. 

Assume $n=2$, and let $\ell=c_1x_1+c_2x_2$ for some $c_1,c_2 \in K$. Recall that $\HF_A(d) \le \HF_A(d+m)$ when $d \le (d_1+d_2-2-m)/2$. We shall prove that when one of the conditions in Proposition \ref{prop:slp_step} fails, we can find a monomial of degree low enough, which is a zero divisor of some power of $\ell$. Write $d_1=m_ip^i+r_i$ and $d_2=n_ip^i+s_i$, for some $i$, as in Proposition \ref{prop:slp_step}, and suppose that condition 1 fails for this $i$. This means that $m_i>0$ and $r_i \le s_i-2$. Then $r_i < d_1$, and therefore $x_1^{r_i} \ne 0$. Recall that
\[\ell^{p^i}=(c_1x_1+c_2x_2)^{p^i}=c_1^{p^i}x_1^{p^i}+c_2^{p^i}x_2^{p^i},\]
since we are in a ring of characteristic $p$. Also,
\[
 (c_1^{p^i}x_1^{p^i}+c_2^{p^i}x_2^{p^i})^{m_i+n_i}=ex_1^{m_ip^i}x_2^{n_ip^i} ~~ \textrm{in}~ A, ~ \textrm{for some} ~ e \in K
\]
since all the other terms in the expansion will be of the form $cx_1^\alpha x_2^\beta$ where either $\alpha \ge d_1$ or $\beta \ge d_2$. 
We have 
\[
 \ell^{(m_i+n_i)p^i}x_1^{r_i}=(c_1^{p^i}x_1^{p^i}+c_2^{p^i}x_2^{p^i})^{m_i+n_i}x_1^{r_i}=ex_1^{m_ip^i}x_2^{n_ip^i}x_1^{r_i}\!=ex_1^{m_ip^i+r_i}x_2^{n_ip^i}\!=0.
\]
In other words, $x_1^{r_i}$ is a monomial in the kernel of the multiplication map $\cdot \ell^{(m_i+n_i)p^i}:A_{r_i} \to A_{r_i+(m_i+n_i)p^i}$, and since
\[
 r_i \le s_i -2 ~~ \iff ~~ r_i \le \frac{r_i+s_i-2}{2} ~~ \iff ~~ r_i \le \frac{d_1+d_2-2-(m_i+n_i)p^i}{2}
\]
we have $\HF_A(r_i) \le \HF_A(r_i+(m_i+n_i)p^i)$. 

If instead conditions 2 of Proposition \ref{prop:slp_step} fails, the proof is carried out in the same way, but with $x_1^{r_i}$ replaced by $x_2^{s_i}$. Suppose now that condition 3 fails for some $i$. That is $m_i,n_i >0$, and $r_i+s_i \le p^i-2$. Then $x_1^{r_i}x_2^{s_i} \ne 0$. We have
\[
 \ell^{(m_i+n_i-1)p^i}=(c_1^{p^i}x_1^{p^i}+c_2^{p^i}x_2^{p^i})^{m_i+n_i-1}=e_1x_1^{(m_i-1)p^i}x_2^{n_ip^i}+e_2x_2^{m_ip^i}x_2^{(n_i-1)p^i}
\]
for some $e_1,e_2 \in K$, and we see that $\ell^{(m_i+n_i-1)p^i}x_1^{r_i}x_2^{s_i} = 0$. Also,
\[
 r_i+s_i \le p^i-2 ~ \iff ~ r_i+s_i \le \frac{r_i+s_i-2+p^i}{2} = \frac{d_1+d_2-2-(m_i+n_i-1)p^i}{2},
\]
which implies that $\HF_A(r_i+s_i) \le \HF_A(r_i+s_i+(m_i+n_i-1)p^i)$.

At last, suppose that condition 4 of Proposition \ref{prop:slp_step} fails. Then $r_i+s_i \ge p^i+2$. This implies that $d_1+d_2-2=m_ip^i+r_i+n_ip^i+s_i \ge (m_i+n_i+1)p^i$, and $\HF((m_i+n_i+1)p^i) \ge 1$. But 
\[
 \ell^{(m_i+n_i+1)p^i}=(c_1^{p^i}x_1^{p^i}+c_2^{p^i}x_2^{p^i})^{m_i+n_i+1}=0,
\]
since all terms in the expansion will be of the form $cx_1^\alpha x_2^\beta$ where either $\alpha \ge d_1$ or $\beta \ge d_2$. This shows that 1 is in the kernel of the multiplication map $\cdot \ell^{(m_i+n_i+1)p^i}: A_{0} \to A_{(m_i+n_i+1)p^i}$. Since $\HF((m_i+n_i+1)p^i) \ge 1 = \HF(0)$, this completes the proof. 
\end{proof}

\section{The Syzygy gap}\label{sec:syz-gap}
The main purpose of this section is to prove Theorem \ref{thm:slp_manhattan}. If we require the residue field to be algebraically closed, the theorem follows from combining a theorem by Han \cite{Han} and results by Brenner and Kaid in \cite{Brenner-Kaid-syzbundle} and \cite{Brenner-lookingout}. Han's result is also proved in a different way by Monsky in \cite{Monsky}. Monsky deals with the syzygy module of three pairwise relatively prime polynomials in two variables, and the so called ''syzygy gap'', while Brenner and Kaid connects this to the Lefschetz properties. We will go through the results from \cite{Monsky}, and give a new proof of the connection to the SLP in the case of monomial complete intersections. The reason to go though the results of \cite{Monsky} is to prove that the residue field does not need to be algebraically closed, but also to give a deeper understanding of Theorem \ref{thm:slp_manhattan} and the theory behind it.  

\subsection{Mason-Stothers' Theorem}
First we need a review of Mason-Stothers' Theorem. Suppose $f$ is a polynomial in $K[x_1, \ldots, x_n]$, where $K$ is some field. The polynomial $f$ can be factorized as $f=\prod_{i=1}^sp_i^{e_i}$, where the $p_i$'s are distinct irreducible factors. Define $\rad(f)=\deg(\prod_{i=1}^sp_i)$. Note that $\rad(fg) \le \rad(f)+\rad(g)$, with equality when $f$ and $g$ are relatively prime. Let $f'_{x_j}$ denote the formal derivative of $f$ w. r. t. the variable $x_j$. When in a polynomial ring with just one variable, we write $f'$ for the derivative. Mason-Stothers' theorem is usually formulated over one variable, as follows. 

\begin{theorem}[Mason-Stothers]\label{thm:Mason1var}
 Let $K$ be a field, and let $f,g$ and $h$ be polynomials in $K[x]$ such that 
 \begin{itemize}
  \item $f,g$ and $h$ are pairwise relatively prime,
  \item $f', g'$ and $h'$ are not all zero,
  \item $f+g+h=0$.
 \end{itemize}
Then $\max(\deg(f), \deg(g), \deg(h)) \le \rad(fgh)-1$.
\end{theorem}
An elementary proof can be found in \cite{Snyder}. There is also a version of this theorem for homogeneous polynomials in two variables. For clairity we will prove how it can be deduced from Theorem \ref{thm:Mason1var}.

\begin{theorem}\label{thm:Mason2var}
 Let $K$ be a field, and let $f, g$ and $h$ be homogeneous polynomials of degree $d$ in $K[x,y]$ such that
 \begin{itemize}
  \item $f,g$ and $h$ are pairwise relatively prime,
  \item $f'_x, f'_y, g'_x, g'_y, h'_x$ and $h'_y$ are not all zero,
  \item $f+g+h=0$.
 \end{itemize}
 Then $d \le \rad(fgh)-2$.
\end{theorem}
\begin{proof}
 Let $K'$ be the splitting field of $f$. Over this field $f$ can be factorized as follows
 \[
  f(x,y) = \sum_{i=0}^d \alpha_ix^iy^{d-i} = y^d \sum_{i=0}^d \alpha_i \Big( \frac{x}{y} \Big)^{\!\! i}=y^d \prod_{j=1}^d\Big(u_j\frac{x}{y}- v_j \Big) = \prod_{j=1}^d(u_ix-v_j y),
 \]
where the $\alpha_i, u_j$ and $v_j$'s are elements in $K'$. After a possible linear change of variables, we can assume that $f(x,y)= y^m\prod_{j=1}^{d-m}(r_jx-s_jy)$, where $m\ge 1$. Let $\hat{f}(x)=f(x,1)=\prod_{j=1}^{d-m}(r_jx-s_j)$, $\hat{g}(x)=g(x,1)$ and $\hat{h}(x)=h(x,1)$. Then $\rad(\hat{f})=\rad(f)-1$, while $\rad(g)=\rad(\hat{g})$ and $\rad(h)=\rad(\hat{h})$. Note also that $\deg(\hat{g})=d$. By Theorem \ref{thm:Mason1var} it now follows that
\[
 d=\deg(\hat{g})\!\le \rad(\hat{f}\hat{g}\hat{h})-1=\rad(\hat{f})+\rad(\hat{g})+\rad(\hat{h})-1 = \rad(f)+\rad(g)+\rad(h)-2=\rad(fgh)-2,
\]
which we wanted to prove.
\end{proof}

\subsection{The syzygy gap}\label{subsec:syz-gap}

Let now $R=K[x,y]$, where $K$ is any field. Let $f_1,f_2$ and $f_3$ be non-zero homogeneous, pairwise relatively prime, polynomials in $R$, with $d_i=\deg(f_i)$, and let $I=(f_1,f_2,f_3)$. The $R$-module $R/I$ has a free resolution of length 2, by Hilbert's syzygy theorem. If $\{f_1, f_2, f_3\}$ is a minimal set of generators of $I$, then 
\begin{equation}\label{resolution}
 0 \to \ker \phi \to R^3 \overset{\phi}{\to}R \to  R/I \to 0,
\end{equation}
where $\phi$ is given by the matrix $\begin{pmatrix} f_1 & f_2 & f_3 \end{pmatrix}$, is an exact sequence of free modules. We have $\rank \ker \phi =3-1=2$. That is, $\ker \phi = \Syz(f_1,f_2,f_3)$ is generated by two homogeneous elements. If $\{f_1, f_2, f_3\}$ is not a minimal set of generators of $I$, we have e. g. $f_3=g_1f_1+g_2f_2$, for some homogeneous polynomials $g_1$ and $g_2$. Then every relation $Af_1+Bf_2+Cf_3=0$ can be written as $(A+Cg_1)f_1+(B+Cg_2)f_2=0$. Since $f_1$ and $f_2$ are relatively prime $A+Cg_1=hf_2$, and $B+Cg_2=-hf_1$, for some homogeneous $h$. It follows that $\ker \phi$ is generated by $ (f_2, -f_1, 0)$ and $(g_1,g_2,-1)$. This shows that (\ref{resolution}) is always a free resolution (but not necessarily minimal), and $\ker \phi$ is generated by two homogeneous elements of degrees, say  $\alpha$ and $\beta$. We have a graded resolution  
\[
 0 \to R(-\alpha) \oplus R(-\beta) \to R(-d_1) \oplus R(-d_2) \oplus R(-d_3) \to R \to R/I \to 0,
\]
of $R/I$. Define $\Delta(f_1,f_2,f_3)=|\alpha-\beta|$. This is the \emph{syzygy gap function} introduced in \cite{Monsky}. From the graded resolution we see that the Hilbert series of $R/I$ is 
\[
 \HS_{R/I}(t)=\frac{1-t^{d_1}-t^{d_2}-t^{d_3}+t^\alpha+t^\beta}{(1-t)^2}.
\]
We also know that $R/I$ has dimension 0, thus the Hilbert series is a polynomial, say $\HS_{R/I}(t)=p(t)$. Then 
\[
 (1-t)^2p(t)=1-t^{d_1}-t^{d_2}-t^{d_3}+t^\alpha+t^\beta.
\]
By taking the derivative of both sides, and substituting $t=1$ we get $0=-d_1-d_2-d_3+\alpha+\beta$, that is $\alpha+\beta=d_1+d_2+d_3$.  
This is one of the so called Herzog-Kühl equations, see e. g. \cite{Progress}. From this follows also the below lemma.

\begin{lemma}\label{lemma:delta_parity}
Let $f_1,f_2$ and $f_3$ be non-zero, pairwise relatively prime homogeneous polynomials in $K[x,y]$, with $d_i=\deg(f_i)$. Then
 $\Delta(f_1,f_2,f_3) \equiv d_1+d_2+d_3 \mod 2$.
\end{lemma}

We shall also see some other properties of the function $\Delta$. 

\begin{lemma}\label{lemma:delta_p^s}
 Let $K$ be a field of characteristic $p>0$, and let $f_1,f_2$ and $f_3$ be non-zero, pairwise relatively prime homogeneous polynomials in $K[x,y]$. Then
 \[\Delta(f_1^{p^s},f_2^{p^s},f_3^{p^s})=p^s\Delta(f_1,f_2,f_3),\] 
 for all non-negative integers $s$.
\end{lemma}

We will give two proofs of this lemma. The first proof uses the Frobenius functor. The second proof is new, and uses elementary methods. 

\begin{proof}[Proof 1]
Let $R=K[x,y]$, and $I=(f_1,f_2,f_3)$. For a fixed $s$, let $q=p^s$. Let $\varphi$ denote the $s$:th power of the Frobenius endomorphism on $R$, that is $\varphi(a)=a^q$. Consider now $R$ as an $R$-bimodule, denoted $R^\varphi$, where left multiplication by an element $r \in R$ is usual multiplication, while right multiplication is multiplication by $\varphi(r)$. The Frobenius functor $\mathcal{F}$ is a functor on the category of left $R$-modules defined by $\mathcal{F}(M)=R^\varphi \otimes_R M$. For a more extensive review of the Frobenius functor, see e. g. \cite{Bruns-Herzog}. Two well knows properties of this functor is that $\mathcal{F}(R^m)\cong R^m$ and $\mathcal{F}(R/I)\cong R/I^{(q)}$ where $I^{(q)}$ is the ideal generated by the $q$:th powers of the elements in $I$. One can also prove that if $\psi: R^m \to R^n$ is a homomorphism of free modules represented by a matrix $(a_{ij})$, then the induced map $\mathcal{F}(\psi)$ is represented by the matrix $(a^q_{ij})$. It follows from \cite[Corollary 2.7]{Kunz} that $\mathcal{F}$ is an exact functor. Now, suppose $\Syz(f_1,f_2,f_3)$ is generated by $(A_1,A_2, A_3)$ and $(B_1, B_2, B_3)$, of degrees $\alpha$ and $\beta$. When we apply $\mathcal{F}$ to the resolution 
\[
 0 \to R^2 
 \overset{\begin{pmatrix}
                     A_1 & B_1 \\
                     A_2 & B_2 \\
                     A_3 & B_3
                    \end{pmatrix}}{\xrightarrow{\hspace*{2cm}}} R^3  \overset{\begin{pmatrix}
                     f_1 & f_2 & f_3
                    \end{pmatrix}}{\xrightarrow{\hspace*{2.5cm}}} R \to R/I \to 0       
\]
we get an exact sequence
\[
 0 \to R^2 
 \overset{\begin{pmatrix}
                     A_1^q & B_1^q \\
                     A_2^q & B_2^q \\
                     A_3^q & B_3^q
                    \end{pmatrix}}{\xrightarrow{\hspace*{2cm}}} R^3  \overset{\begin{pmatrix}
                     f_1^q & f_2^q & f_3^q
                    \end{pmatrix}}{\xrightarrow{\hspace*{2.5cm}}} R \to R/I^{(p)} \to 0.
\]
This proves that $\Syz(f_1^q,f_2^q,f_3^q)$ is generated by $(A_1^q,A_2^q,A_3^q)$ and $(B_1^q,B_2^q,B_3^q)$. Then 
 \[
  \Delta(f_1^{q},f_2^q,f_3^{q})=|aq-bq|=q|a-b|=q\Delta(f_1,f_2,f_3),
 \]
which we wanted to prove.
\end{proof}

\begin{proof}[Proof 2]
 Suppose $\Syz(f_1,f_2,f_3)$ is generated by $(A_1,A_2, A_3)$ and $(B_1, B_2, B_3)$, of degrees $\alpha$ and $\beta$. Notice first that
\[
 A_1f_1+A_2f_2+A_3f_3=0 ~~ \iff ~~ A_1^{p^s}f_1^{p^s}+A_2^{p^s}f_2^{p^s}+A_3^{p^s}f_3^{p^s}=0,
\]
and the corresponding for $(B_1,B_2,B_3)$, since $(A_1f_1+A_2f_2+A_3f_3)^{p^s}=A_1^{p^s}f_1^{p^s}+A_2^{p^s}f_2^{p^s}+A_3^{p^s}f_3^{p^s}$. However, 
this does not prove that $\Syz(f_1^{p^s},f_2^{p^s},f_3^{p^s})$ is generated by $(A_1^{p^s},A_2^{p^s}, A_3^{p^s})$ and $(B_1^{p^s}, B_2^{p^s}, B_3^{p^s})$.

Claim: The two syzygies generating $\Syz(f_1^{p^s},f_2^{p^s},f_3^{p^s})$ consist of polynomials in $x^{p^s}$ and $y^{p^s}$. 

These two generators give two syzygies in $\Syz(f_1,f_2,f_3)$ by the change of variables $x^{p^s} \mapsto x$ and $y^{p^s} \mapsto y$. It follows that the two relations that generate $\Syz(f_1^{p^s},f_2^{p^s},f_3^{p^s})$ has degrees $\alpha p^s$ and $\beta p^s$. Then 
 \[
  \Delta(f_1^{p^s},f_2^{p^s},f_3^{p^s})=|ap^s-bp^s|=p^s|a-b|=p^s\Delta(f_1,f_2,f_3),
 \]
which we wanted to prove.

It remains to prove the claim. We will first prove that the two syzygies consist of polynomials in $x^p$ and $y^p$. Assume that
\begin{equation}\label{eq:delta_p^s:syz}
 C_1f_1^{p^s}+C_2f_2^{p^s}+C_3f_3^{p^s}=0 ~~ \textrm{and} ~~ D_1f_1^{p^s}+D_2f_2^{p^s}+D_3f_3^{p^s}=0
\end{equation}
are the two relations that generate $\Syz(f_1^{p^s},f_2^{p^s},f_3^{p^s})$. We may assume that $(C_1,C_2,C_3)$ is the one of lowest degree. Let $f'$ denote the derivative of $f$ w. r. t. $x$. Notice that 
\[
(Cf^{p^s})'=C'f^{p^s}+p^sCf^{p^s-1}f'=C'f^{p^s},
\]
and hence
\[
 C_1'f_1^{p^s}+C_2'f_2^{p^s}+C_3'f_3^{p^s}=0. 
\]
But there can not be a relation of lower degree than $(C_1,C_2,C_3)$, thus $C_i'=0$ for $i=1,2,3$. The same holds when we take the derivative w. r. t. $y$. This means that $C_1, C_2, C_3$ are polynomials in $x^p$ and $y^p$. 

If we derivate the other relation w. r. t. $x$ we get 
\[
 D_1'f_1^{p^s}+D_2'f_2^{p^s}+D_3'f_3^{p^s}=0.
\]
Here we can not conclude that all $D_i'=0$, because there are elements of lower degree, namely those generated by $(C_1,C_2,C_3)$. Suppose there is a $g$ such that $D_i'=gC_i$, for $i=1,2,3$. Recall that $C_i'=0$. This gives $D_i=GC_i+\widehat D_i$, where $G$ is some homogeneous polynomial such that $G'=g$, and $\widehat D_i'=0$. We can assume that $(\widehat D_1, \widehat D_2, \widehat D_3)$ is not a multiple of $(C_1,C_2,C_3)$, because that would imply that $(D_1,D_2,D_3)$ is also a multiple of $(C_1,C_2,C_3)$. Therefore we can replace $(D_1,D_2,D_3)$ by $(\widehat D_1, \widehat D_2, \widehat D_3)$. Since $\widehat D_i'=0$, $\widehat D_i$ is a homogeneous polynomial in $x^p$ and $y$. Suppose $\widehat D_i$ has degree $m_ip+r_i$, where $0 \le r_i <p$. Then the terms in $\widehat D_i$ looks like $cx^{\alpha_ip}y^{(m_i-\alpha_i)p+r_i}$, which means that $\widehat D_i$ is divisible by $y^{r_i}$. Also $\deg(\widehat D_if_i^{p^s})=m_ip+d_ip^s+r_i$. Since 
\[
 \deg(\widehat D_1f_1^{p^s})=\deg(\widehat D_2f_2^{p^s})=\deg(\widehat D_3f_3^{p^s})
\]
$r_1=r_2=r_3$. If $r_1>0$ we can divide $\widehat D_1, \widehat D_2$ and $\widehat D_3$ by $y^{r_1}$ and get a syzygy of lower degree, which is not a multiple of $(C_1,C_2,C_3)$. But this is a contradiction, so $r_1=r_2=r_3=0$. This means that $\widehat D_1, \widehat D_2$ and $\widehat D_3$ also are polynomials in $x^p$ and $y^p$. A change of variables $x^p \mapsto x$ and $y^p \mapsto y$ in (\ref{eq:delta_p^s:syz}) gives two elements in $\Syz(f_1^{p^{s-1}},f_2^{p^{s-1}},f_3^{p^{s-1}})$. The claim now follows by induction. 
\end{proof}

 Let us now investigate what happens with $\Delta(f_1,f_2,f_3)$ when, for example, $f_1$ is replaced by $\ell f_1$, for some linear form $\ell$. By Lemma \ref{lemma:delta_parity}, $\Delta(f_1,f_2,f_3)$ and $\Delta(\ell f_1,f_2,f_3)$ has different parity, so they can not be equal. If we have a relation $A_1f_1+A_2f_2+A_3f_3=0$, we also get a relation on $\ell f_1, f_2, f_3$ by multiplying the expression by $\ell$. This means that the two elements that generates $\Syz(\ell f_1,f_2,f_3)$ can have degrees at most $\alpha+1$ and $\beta+1$. On the other hand, a relation $A_1\ell f_1+A_2f_2+A_3f_3=0$ on $\ell f_1, f_2, f_3$ can also be considered a syzygy $(A_1\ell, A_2, A_3)$ on $f_1,f_2,f_3$. Hence, the two generators of $\Syz(\ell f_1,f_2,f_3)$ have degrees at least $\alpha$ and $\beta$. This shows that $\Delta$ must either increase of decrease by 1 when $f_1$ is replaced by $\ell f_1$. We summarize this in a lemma. 
 
 \begin{lemma}\label{lemma:delta_plusminus1}
 Let $f_1,f_2$ and $f_3$ be non-zero, pairwise relatively prime homogeneous polynomials in $K[x,y]$. Let $\ell$ be a linear form, relatively prime to $f_2$ and $f_3$. Then
 \[\Delta(\ell f_1,f_2,f_3)=\Delta(f_1,f_2,f_3) \pm 1.\]
 \end{lemma}

We shall look more carefully into two special cases where Lemma \ref{lemma:delta_plusminus1} applies. Let $(A_1, A_2, A_3)$ be the element in $\Syz(f_1,f_2,f_3)$ of the lowest degree $\alpha$. If $\ell | A_1$ then $(\ell^{-1}A_1 ,A_2,A_3)$ is a syzygy of $\ell f_1,f_2,f_3$ of degree $\alpha$. The other generating syzygy can have degree $\beta$ or $\beta+1$, as we saw above. But since $\Delta(\ell f_1,f_2,f_3)\ne \Delta(f_1,f_2,f_3)$ it must have degree $\beta+1$. Hence, $\Delta(\ell f_1,f_2,f_3)=\Delta(f_1,f_2,f_3) + 1$ in this case.

It follows also from Lemma \ref{lemma:delta_plusminus1} that $\Delta(\ell^{-1} f_1,f_2,f_3)=\Delta(f_1,f_2,f_3) \pm 1$, if $\ell | f_1$. If, in addition, $\ell | A_2$, it follows from the equality $A_1f_1+A_2f_2+A_3f_3=0$ that $\ell$ also divides $A_3$. Then we can divide the whole expression by $\ell$, and get a syzygy $(A_1,\ell^{-1}A_2, \ell^{-1}A_3)$ on $\ell^{-1}f_1, f_2, f_3$, of degree $\alpha-1$. We see that we must have $\Delta(\ell^{-1} f_1,f_2,f_3)=\Delta(f_1,f_2,f_3) + 1$, in this case. 

This, together with Theorem \ref{thm:Mason2var}, can now be used to prove the following proposition. 

\begin{prop}[{\cite[Theorem 8]{Monsky}}]\label{thm8monsky}
 Let $K$ be a field of characteristic $p>0$. Let $f_1, f_2,$ and $f_3$ be homogeneous relatively prime polynomials in $K[x,y]$. Assume there is a linear form $\ell$ such that $f_1=\ell^m h$, where $l \nmid h$ and $p \nmid m$. Assume also that $\Delta(f_1,f_2,f_3)$ decreases when $f_1$ is replaced by $\ell f_1$ or $\ell^{-1}f_1$. Then $\Delta(f_1,f_2,f_3) \le \rad(f_1f_2f_3)-2$.
\end{prop}
\begin{proof}
  Let $(A_1,A_2,A_3)$ be one of the two generators of $\Syz(f_1,f_2,f_3)$ of minimal degree $\alpha$. We saw above that if $\ell | A_1$ then $\Delta(\ell f_1,f_2,f_3)=\Delta(f_1,f_2,f_3) + 1$. We also saw that if $\ell |A_2$ then $\Delta(\ell^{-1} f_1,f_2,f_3)=\Delta(f_1,f_2,f_3) + 1$. The same holds if $\ell | A_3$. By assumption, none of this is the case, and hence $A_1, A_2$ and $A_3$ are not divisible $\ell$. Let $M=\gcd(A_1f_1,A_2f_2,A_3f_3)$. Then 
 \[
  \frac{A_1f_1}{M}+\frac{A_2f_2}{M}+\frac{A_3f_3}{M}=0
 \]
and the three terms $A_if_i/M$ are relatively prime. Notice that every irreducible factor of $M$ must divide one of $f_1, f_2$ or $f_3$. Also $\ell$ does not divide $M$, since $\ell$ does not divide $A_2,A_3,f_2$ or $f_3$. We shall now see that the formal derivative of $A_1f_1/M$ w. r. t. $x$ or $y$ is non-zero, so that we can use Theorem \ref{thm:Mason2var}. One of $\ell'_x$ and $\ell'_y$ must be non-zero, otherwise $\ell=0$. Say that $\ell'_x=c \ne 0$. Then
\[
 \Big(\frac{A_1f_1}{M} \Big)'_x =  \Big(\ell^m\frac{A_1h}{M} \Big)'_x = mc\ell^{m-1} \frac{A_1h}{M} + \ell^m  \Big(\frac{A_1h}{M} \Big)'_x.
\]
The two terms can not cancel each other, and the first one is non-zero, since $m \ne 0$ in $K$. Hence $(A_1f_1/M)'_x \ne 0$. By Theorem \ref{thm:Mason2var} 
\begin{equation}\label{eq:thm8monsky}
 \deg\Big( \frac{A_1f_1}{M} \Big) \le \rad\Big(\frac{A_1f_1A_2f_2A_3f_3}{M^3}\Big)-2.
\end{equation}
We know that $\deg(A_1f_1/M)=\alpha-\deg(M)$. Let $d_i=\deg(f_i)$, for $i=1,2,3$, and recall that $d_1+d_2+d_3=\alpha+\beta$, where $\beta$ is the degree of the other generator of $\Syz(f_1,f_2,f_3)$. We have
\begin{align*}
 \rad\!\Big(\!\frac{A_1f_1A_2f_2A_3f_3}{M^3}\!&\Big) \le \rad\Big(\frac{f_1f_2f_3}{M}\Big) + \deg\Big(\frac{A_1A_2A_3}{M^2}\Big) \\
 &\le \rad(f_1f_2f_3) + \deg(A_1)+\deg(A_2)+\deg(A_3)-2\deg(M) \\
 &=\rad(f_1f_2f_3) + (\alpha-d_1)+(\alpha-d_2)+(\alpha-d_3)-2\deg(M) \\
 &=\rad(f_1f_2f_3)+3\alpha-(\alpha+\beta)-2\deg(M) \\
 &= \rad(f_1f_2f_3)+2\alpha-\beta-2\deg(M) .
\end{align*}
 Inserted in (\ref{eq:thm8monsky}), this gives
\[
 \alpha-\deg(M) \le \rad(f_1f_2f_3)+2\alpha-\beta-2\deg(M)-2,
\]
which is rewritten as
\[
 \beta-\alpha \le \rad(f_1f_2f_3)-\deg(M)-2.
\]
We can now conclude that $\Delta(f_1,f_2,f_3)=\beta-\alpha\le  \rad(f_1f_2f_3)-2.$
\end{proof}

\subsection{Application of the syzygy gap function to monomial complete intersection algebras}
We will now specialize to the case $f_1=x^{d_1}$, $f_2=y^{d_2}$, and $f_3=(x+y)^{d_3}$. This is allowed, since these polynomials are pairwise relatively prime. For an easier notation we introduce a new function $\delta:\mathbb{Z}_+^3 \to \mathbb{Z}_{\ge 0}$ defined by $\delta(d_1,d_2,d_3)=\Delta(x^{d_1}, y^{d_2}, (x+y)^{d_3})$. We will now see how the theory of the syzygy gap connects to the SLP. 

\begin{prop}\label{prop:maxrang-delta}
 Let $S=K[x,y]/(x^{d_1}, y^{d_2})$. The maps $\cdot (x+y)^{d_3}: S_i \to S_{i+d_3}$, with $d_3 <d_1+d_2$, have maximal rank for all $i$ if and only if $\delta(d_1,d_2,d_3) \le 1$. 
\end{prop}

This result can be proved for general $f_1, f_2$ and $f_3$ using \cite[Theorem 2.2]{Brenner-Kaid-syzbundle} and \cite[Corollary 3.2]{Brenner-lookingout}. Below follows an easier proof for this special case. 

\begin{proof}
 We know that the syzygy module $\Syz(x^{d_1}, y^{d_2}, (x+y)^{d_3})$ is generated by two homogenous elements $(A_1,A_2,A_3)$ and $(B_1,B_2,B_3)$ of degrees $\alpha$ and $\beta$. We may assume that $\alpha \le \beta$. Provided that $A_3 \ne 0$, this can be formulated as $(x+y)^{d_3}A_3 =0$ in $S$, and $A_3$ is a homogenous element of lowest degree with this property. The degree of $A_3$ is $\alpha-d_3$. By Proposition \ref{prop:maxrang-inj} multiplication by $(x+y)^{d_3}$ has maximal rank in every degree if and only if
 \[
  \alpha-d_3 > \frac{d_1+d_2-2-d_3}{2} ~~ \textrm{or equivalently} ~~ \alpha >\frac{d_1+d_2+d_3-2}{2}.
 \]
Recall that $\alpha+\beta=d_1+d_2+d_3$. This inserted in the above inequality gives, after simplification, $\alpha > \beta-2$. Since $\alpha \le \beta$ this is exactly the property $\delta(d_1,d_2,d_3)=\beta-\alpha \le 1. $  

It remains to prove that $A_3 \ne 0$. If $A_3=0$ we would have a relation $A_1f_1+A_2f_2=0$. Since $f_1$ and $f_2$ are relatively prime, this gives $A_1=cf_2$ and $A_2=-cf_1$, for some $c \in K$. Then $\alpha=d_1+d_2$, and since $\alpha+\beta=d_1+d_2+d_3$, we get $\beta=d_3$. But $\beta \ge \alpha$ and $d_3 < d_1+d_2$ yields a contradiction. 
\end{proof}

This result combined with Proposition \ref{prop:slp-maxrang} now gives the following. 

\begin{theorem}\label{thm:slp_delta=0}
 The algebra $K[x,y]/(x^{d_1}, y^{d_2})$ has the SLP if and only if \[\delta(d_1,d_2,d_1+d_2-2c)=0 ~~ \textrm{for all} ~~ 1 \le c < \min(d_1,d_2).\] 
\end{theorem}
\begin{proof}
 It follows directly from Proposition \ref{prop:maxrang-delta} and Proposition \ref{prop:slp-maxrang} that $K[x,y]/(x^{d_1}, y^{d_2})$ has the SLP if and only if $\delta(d_1,d_2,d_1+d_2-2c)\le 1$. By Lemma \ref{lemma:delta_parity}  $\delta(d_1,d_2,d_1+d_2-2c)$ is even, so it must be 0 in this case. 
\end{proof}

The problem now is to determine for which $d_1,d_2,d_3$ we have $\delta(d_1,d_2,d_3)=0$. Let us define
 \[
  L=\{(u,v,w)\in \mathbb{Z}_+^3 ~|~ 2\max(u,v,w) \le u+v+w\}.
 \]
Also, let $L_=$ be the subset of $L$ where equality holds, and $L_<=L \setminus L_=$.
\begin{lemma}\label{lemma:delta=0_on_L_=}
 Let $(d_1,d_2,d_3) \in L_=$. Then $\delta(d_1,d_2,d_3)=0$.
\end{lemma}
\begin{proof}
 Suppose $d_1 \le d_2 <d_3=d_1+d_2$. We are in the situation when $x^{d_1}$, $y^{d_2}$, $(x+y)^{d_3}$ is not a minimal generating set; there are polynomials $g$ and $h$ such that $
 (x+y)^{d_1+d_2}=gx^{d_1}+hy^{d_2}$
 As we saw in the beginning of Section \ref{subsec:syz-gap}, the module $\Syz(x^{d_1},y^{d_2},(x+y)^{d_1+d_2})$ is, in this case, generated by $(g,h,-1)$ and $(y^{d_2}, -x^{d_1}, 0)$. Both these relations have degree $d_1+d_2$, which gives $\delta(d_1,d_2,d_3)=0$. 
 
 The case when $d_1$ or $d_2$ is the largest among $d_1,d_2,d_3$ follows from the above after a linear change of the variables $x$ and $y$. 
 
\end{proof}

\begin{lemma}\label{lemma:delta_ineq}
 For any two points $(c_1,c_2,c_3)$ and $(d_1,d_2,d_3)$ in $\mathbb{Z}_+^3$ it holds that
 \begin{equation}\label{eq:lemma:delta_ineq}
  | \delta(c_1,c_2,c_3) - \delta(d_1,d_2,d_3)| \le |c_1-d_1|+|c_2-d_2|+|c_3-d_3|.
 \end{equation}
Moreover, for $(d_1,d_2,d_3)\in L_<$ we can find a point $(c_1,c_2,c_3)$ such that 
\[
 \delta(c_1,c_2,c_3) = \delta(d_1,d_2,d_3) + |c_1-d_1|+|c_2-d_2|+|c_3-d_3|,
\]
and $\delta(c_1,c_2,c_3)$ decreases when any $c_i$ is replaced by $c_i \pm 1$.
\end{lemma}
\begin{proof}
 Recall from Lemma \ref{lemma:delta_plusminus1} that $\delta(d_1,d_2,d_3)$ increases or decreases by 1 when we ''take a step'' in $\mathbb{Z}_+^3$, that is when one $d_i$ is replaced by $d_i \pm 1$. This proves (\ref{eq:lemma:delta_ineq}). 
 
 Imagine now that we start in the point $(d_1,d_2,d_3)$, and take a step in some direction, if it makes the value of $\delta$ increase. We continue in this way, as long as we can make the value of $\delta$ increase in each step. What we want to prove is that such a path can not be infinitely long. Let us fix a point $(d'_1,d'_2,d'_3)$ on our path. Any other path between $(d_1,d_2,d_3)$ and $(d'_1,d'_2,d'_3)$ must give the same value of $\delta$ at $(d'_1,d'_2,d'_3)$. It follows that a path where the value of $\delta$ increases in each step must be of minimal length, among all paths between these two points. Any other path of minimal length must also have the property that $\delta$ increases in each step. Hence we can replace our path by the path that first increases/decreases $d_1$, then $d_2$ and last $d_3$. But when $d_2$ and $d_3$ are fixed, we can only increase of decrease $d_1$ a finite number of times, before we hit $L_=$. The corresponding holds for $d_2$ and $d_3$. At $L_=$ the value of $\delta$ is zero, as we saw in Lemma \ref{lemma:delta=0_on_L_=}, so $\delta$ must have decreased. This shows that there is a bound for the length of a path that starts in a given point $(d_1,d_2,d_3)\in L_<$, and increases $\delta$ in each step. Eventually we will reach a point $(c_1,c_2,c_3)$ such that
 \[
 \delta(c_1,c_2,c_3) = \delta(d_1,d_2,d_3) + |c_1-d_1|+|c_2-d_2|+|c_3-d_3|,
\]
and $\delta(c_1,c_2,c_3)$ decreases when any $c_i$ is replaced by $c_i \pm 1$.
\end{proof}

\begin{center}
\definecolor{qqqqff}{rgb}{0,0,1}
\definecolor{xdxdff}{rgb}{0.49,0.49,1}
\definecolor{cqcqcq}{rgb}{0.75,0.75,0.75}

\begin{tikzpicture}[scale =0.46,line cap=round,line join=round,>=triangle 45,x=1.0cm,y=1.0cm]
\draw [color=cqcqcq,dash pattern=on 3pt off 3pt, xstep=1.0cm,ystep=1.0cm] (0,0) grid (22.12,18.5);
\draw[->,color=black] (0,0) -- (22.12,0);
\foreach \x in {,2,4,6,8,10,12,14,16,18,20,22}
\draw[shift={(\x,0)},color=black] (0pt,2pt) -- (0pt,-2pt);
\draw[->,color=black] (0,0) -- (0,18.5);
\foreach \y in {,2,4,6,8,10,12,14,16,18}
\draw[shift={(0,\y)},color=black] (2pt,0pt) -- (-2pt,0pt);
\clip(-2.26,-1.39) rectangle (22.12,18.5);
\draw (10,0)-- (0,9.96);
\draw [domain=10.0:25.119256353201717] plot(\x,{(-10--1*\x)/1});
\draw [domain=0.0:25.119256353201717] plot(\x,{(--9.96--1.0*\x)/1});
\draw [line width=1.1pt] (9,9)-- (9,10);
\draw [line width=1.1pt] (9,10)-- (10,10);
\draw [line width=1.1pt] (10,10)-- (11,10);
\draw [line width=1.1pt] (11,10)-- (11,11);
\draw [line width=1.1pt] (12,11)-- (11,11);
\draw [line width=1.1pt] (12,11)-- (12,12);
\draw [line width=1.1pt] (12,12)-- (13,12);
\draw [line width=1.1pt] (13,12)-- (13,13);
\draw [line width=1.1pt] (9,9)-- (9,8);
\draw [line width=1.1pt] (9,8)-- (9,7);
\draw (9,7)-- (13,7);
\draw (13,7)-- (13,13);
\begin{scriptsize}
\draw (8.8,7.72) node[anchor=north west] {$ _{+1} $};
\draw (8.8,8.72) node[anchor=north west] {$ _{+1} $};
\draw (8.8,9.72) node[anchor=north west] {$ _{+1} $};
\draw (9.1,10.1) node[anchor=north west] {$ _{+1} $};
\draw (10.1,10.09) node[anchor=north west] {$  _{+1}$};
\draw (10.8,10.72) node[anchor=north west] {$  _{+1}$};
\draw (11.1,11.1) node[anchor=north west] {$ _{+1} $};
\draw (11.8,11.72) node[anchor=north west] {$ _{+1} $};
\draw (12.1,12.1) node[anchor=north west] {$ _{+1} $};
\draw (12.8,12.72) node[anchor=north west] {$ _{+1} $};
\draw (9.0,7.1) node[anchor=north west] {$ _{+1} $};
\draw (10,7.1) node[anchor=north west] {$ _{+1} $};
\draw (6.98,14.89) node[anchor=north west] {\scalebox{2}{ $L$} };
\draw (11.1,7.1) node[anchor=north west] {$ _{+1} $};
\draw (12.1,7.1) node[anchor=north west] {$ _{+1} $};
\draw (9.5,0) node[anchor=north west] {$ d_3 $};
\draw (-1.,10.4) node[anchor=north west] {$ d_3 $};
\draw (12.8,7.9) node[anchor=north west] {$ _{+1} $};
\draw (12.8,8.9) node[anchor=north west] {$ _{+1} $};
\draw (12.8,11.9) node[anchor=north west] {$ _{_{+1}} $};
\draw (12.8,10.9) node[anchor=north west] {$ _{+1} $};
\draw (12.8,9.9) node[anchor=north west] {$ _{+1} $};
\draw (18,8.07) node[anchor=north west] {$ d_1=d_2+d_3 $};
\draw (3.5,18) node[anchor=north west] {$ d_2=d_1+d_3 $};
\draw (17,7.2) node[anchor=north west] {$ \delta=0 $};
\draw (4.5,17) node[anchor=north west] {$ \delta=0 $};
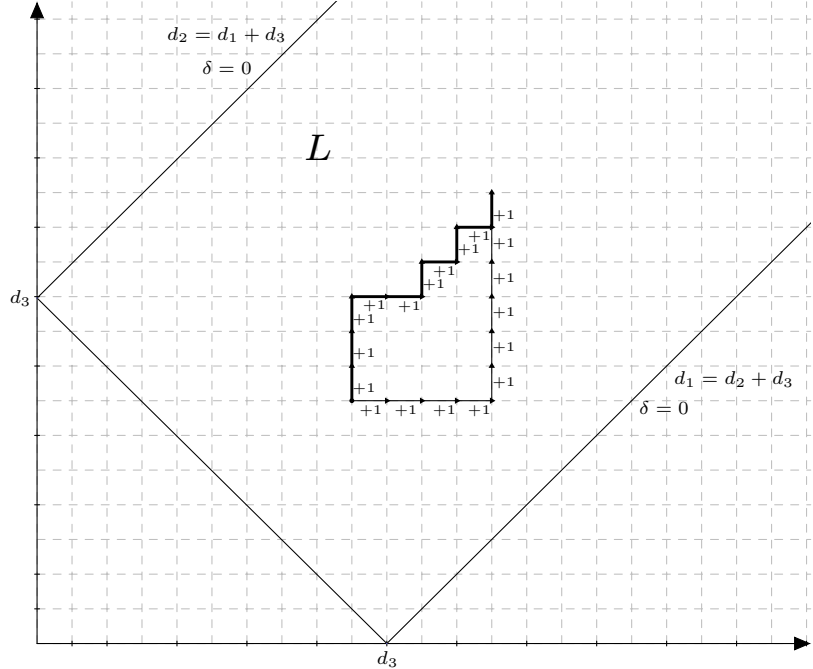
\draw (0,-1) node[anchor= west] {The set $L$, and a path with increasing $\delta$, for a fixed $d_3$.};
\fill [color=xdxdff] (10,0) circle (0.5pt);
\fill [color=xdxdff] (0,9.96) circle (0.5pt);

\fill [color=black,shift={(9,9)}] (0,0) ++(0 pt,3.0pt) -- ++(2.6pt,-4.5pt)--++(-5.2pt,0 pt) -- ++(2.6pt,4.5pt);
\fill [color=black,shift={(9,10)}] (0,0) ++(0 pt,3.0pt) -- ++(2.6pt,-4.5pt)--++(-5.2pt,0 pt) -- ++(2.6pt,4.5pt);
\fill [color=black,shift={(10,10)},rotate=270] (0,0) ++(0 pt,3.0pt) -- ++(2.6pt,-4.5pt)--++(-5.2pt,0 pt) -- ++(2.6pt,4.5pt);
\fill [color=black,shift={(11,10)},rotate=270] (0,0) ++(0 pt,3.0pt) -- ++(2.6pt,-4.5pt)--++(-5.2pt,0 pt) -- ++(2.6pt,4.5pt);
\fill [color=black,shift={(11,11)}] (0,0) ++(0 pt,3.0pt) -- ++(2.6pt,-4.5pt)--++(-5.2pt,0 pt) -- ++(2.6pt,4.5pt);
\fill [color=black,shift={(12,11)},rotate=270] (0,0) ++(0 pt,3.0pt) -- ++(2.6pt,-4.5pt)--++(-5.2pt,0 pt) -- ++(2.6pt,4.5pt);
\fill [color=black,shift={(12,12)}] (0,0) ++(0 pt,3.0pt) -- ++(2.6pt,-4.5pt)--++(-5.2pt,0 pt) -- ++(2.6pt,4.5pt);
\fill [color=black,shift={(13,12)},rotate=270] (0,0) ++(0 pt,3.0pt) -- ++(2.6pt,-4.5pt)--++(-5.2pt,0 pt) -- ++(2.6pt,4.5pt);
\fill [color=black,shift={(13,13)}] (0,0) ++(0 pt,3.0pt) -- ++(2.6pt,-4.5pt)--++(-5.2pt,0 pt) -- ++(2.6pt,4.5pt);
\fill [color=black,shift={(9,8)}] (0,0) ++(0 pt,3.0pt) -- ++(2.6pt,-4.5pt)--++(-5.2pt,0 pt) -- ++(2.6pt,4.5pt);
\fill [color=black] (9,7) circle (2.0pt);
\fill [color=black,shift={(13,7)},rotate=270] (0,0) ++(0 pt,3.0pt) -- ++(2.6pt,-4.5pt)--++(-5.2pt,0 pt) -- ++(2.6pt,4.5pt);
\fill [color=black,shift={(10,7)},rotate=270] (0,0) ++(0 pt,3.0pt) -- ++(2.6pt,-4.5pt)--++(-5.2pt,0 pt) -- ++(2.6pt,4.5pt);
\fill [color=black,shift={(11,7)},rotate=270] (0,0) ++(0 pt,3.0pt) -- ++(2.6pt,-4.5pt)--++(-5.2pt,0 pt) -- ++(2.6pt,4.5pt);
\fill [color=black,shift={(12,7)},rotate=270] (0,0) ++(0 pt,3.0pt) -- ++(2.6pt,-4.5pt)--++(-5.2pt,0 pt) -- ++(2.6pt,4.5pt);
\fill [color=black,shift={(13,8)}] (0,0) ++(0 pt,3.0pt) -- ++(2.6pt,-4.5pt)--++(-5.2pt,0 pt) -- ++(2.6pt,4.5pt);
\fill [color=black,shift={(13,9)}] (0,0) ++(0 pt,3.0pt) -- ++(2.6pt,-4.5pt)--++(-5.2pt,0 pt) -- ++(2.6pt,4.5pt);
\fill [color=black,shift={(13,10)}] (0,0) ++(0 pt,3.0pt) -- ++(2.6pt,-4.5pt)--++(-5.2pt,0 pt) -- ++(2.6pt,4.5pt);
\fill [color=black,shift={(13,11)}] (0,0) ++(0 pt,3.0pt) -- ++(2.6pt,-4.5pt)--++(-5.2pt,0 pt) -- ++(2.6pt,4.5pt);

\end{scriptsize}
\end{tikzpicture}
\end{center}
%

\begin{theorem}\label{thm:delta=0_kriterium}
Let the function $\delta$ be defined over $K[x,y]$ where $K$ is a field of characteristic $p>0$. Let $d_1,d_2,d_3$ be positive integers, such that $d_1 \le d_2 \le d_3 <d_1+d_2$. Then $\delta(d_1,d_2,d_3)=0$ if and only if
\[|d_1-up^s|+|d_2-vp^s|+|d_3-wp^s| \ge p^s \]
for all integers $s,u,v,w$ such that $s \ge 0 $ and $u+v+w$ is odd.
\end{theorem}
\begin{proof}
 Assume first that $\delta(d_1,d_2,d_3)=0$. Let $s$ be a non-negative integer, and $u,v,w$ integers with odd sum. From Lemma \ref{lemma:delta_parity} we know that $\delta(u,v,w)$ is odd, in particular $\delta(u,v,w)\ge 1$. Lemma \ref{lemma:delta_p^s} and Lemma \ref{lemma:delta_ineq} now gives
 \begin{align*}
 |d_1-up^s|+|d_2-vp^s|+|d_3-wp^s| & \ge | \delta(up^s,vp^s,wp^s)-\delta(d_1,d_2,d_3)| \\
 &= \delta(up^s,vp^s,wp^s) = p^s \delta(u,v,w) \ge p^s. 
 \end{align*}
For the other implication, assume that 
\[|d_1-up^s|+|d_2-vp^s|+|d_3-wp^s| \ge p^s \]
for all $s \ge 0$ and integers $u,v,w$ with odd sum. By Lemma \ref{lemma:delta_ineq} there is a point $(c_1,c_2,c_3)$ such that 
\[ \delta(c_1,c_2,c_3)=\delta(d_1,d_2,d_3)+ |d_1-c_1|+|d_2-c_2|+|d_3-c_3|, \]
and $\delta(c_1,c_2,c_3)$ decreases by one if we replace any $c_i$ by $c_i \pm 1$. Write $c_1=p^su$, $c_2=p^sv$ and $c_3=p^sw$, such that (at least) one of $u$, $v$, and $w$ is not divisible by $p$. Notice that $\delta(u,v,w)$ also must decrease when $u,v$ or $w$ is increased or degreased by one. Otherwise we would have e. g. $\delta(u,v,w+1)=\delta(u,v,w)+1$, which implies $\delta(c_1,c_2,c_3+p^s)=\delta(c_1,c_2,c_3)+p^s$. This can only hold if $\delta$ increases in each step from $(c_1,c_2,c_3)$ to $(c_1,c_2,c_3+p^s)$, which is not the case. Now we can use Proposition \ref{thm8monsky} on $\delta(u,v,w)$ with $\ell=x,y,$ or $x+y$, depending on which of $u,v$ and $w$ are not divisible by $p$. Since $\rad(x^uy^v(x+y)^w)=3$ we get $\delta(u,v,w)\le 1$. Since $\delta(u,v,w)-1=\delta(u,v,w+1)\ge 0$, we must have $\delta(u,v,w)=1$. By Lemmma \ref{lemma:delta_parity} $u+v+w$ is odd, and we can use our assumption to get
\begin{align*}
\delta(d_1,d_2,d_3)&=\delta(c_1,c_2,c_3)-(|d_1-c_1|+|d_2-c_2|+|d_3-c_3|)\\
&=p^s\delta(u,v,w)-(|d_1-up^s|+|d_2-vp^s|+|d_3-wp^s|)\\
&=p^s-(|d_1-up^s|+|d_2-vp^s|+|d_3-wp^s|) \le 0.
\end{align*}
By definition $\delta(d_1,d_2,d_3) \ge 0$, so we can conclude $\delta(d_1,d_2,d_3)=0$.
\end{proof}

\begin{proof}[Proof of Theorem \ref{thm:slp_manhattan}]
 By Theorem \ref{thm:slp_delta=0}, $K[x,y]/(x^{d_1},y^{d_2})$ has the SLP if and only if
 \[
  \delta(d_1,d_2,d_1+d_2-2c)=0 \ \mbox{for all} \ 1 \le c < \min(d_1,d_2). 
 \]
With $d_3=d_1+d_2-2c$, clearly $d_1 \le d_3$, $d_2 \le d_3$ and $d_3 <d_1+d_2$, so we can use Theorem \ref{thm:delta=0_kriterium}. Substituting $d_3=d_1+d_2-2c$ into the inequality in Theorem \ref{thm:delta=0_kriterium}, gives Theorem \ref{thm:slp_manhattan}. 
\end{proof}

\end{document}